\documentclass[11 pt]{article}
\usepackage[top=1in, bottom=1in, left=1in, right=1in]{geometry}
\usepackage{amsmath}
\usepackage{amsthm}
\usepackage{amsfonts}
\usepackage{amssymb}
\usepackage{mathrsfs}
\usepackage{url}
\usepackage{enumerate}
\usepackage[shortlabels]{enumitem}
\usepackage{multirow}
\usepackage{float}
\usepackage{mathtools}
\usepackage{tikz-cd}

\DeclareMathOperator{\Frob}{Frob}

\DeclareMathOperator{\End}{End}
\DeclareMathOperator{\Gal}{Gal}

\DeclareMathOperator{\GL}{GL}
\DeclareMathOperator{\SL}{SL}

\DeclareMathOperator{\PGL}{PGL}

\DeclareMathOperator{\Tr}{Tr}

\DeclareMathOperator{\Hom}{Hom}

\DeclareMathOperator{\Ann}{Ann}
\DeclareMathOperator{\Ker}{Ker}
\DeclareMathOperator{\an}{an}
\DeclareMathOperator{\Katz}{Katz}

\newcommand{\C}{\mathbb{C}}
\newcommand{\Z}{\mathbb{Z}}
\newcommand{\Q}{\mathbb{Q}}
\newcommand{\F}{\mathbb{F}}
\newcommand{\T}{\mathbb{T}}

\newcommand{\tensor}{\otimes}

\newtheorem{innercustomgeneric}{\customgenericname}
\providecommand{\customgenericname}{}
\newcommand{\newcustomtheorem}[2]{%
  \newenvironment{#1}[1]
  {%
   \renewcommand\customgenericname{#2}%
   \renewcommand\theinnercustomgeneric{##1}%
   \innercustomgeneric
  }
  {\endinnercustomgeneric}
}

\newcustomtheorem{customthm}{Theorem}
\newcustomtheorem{customlemma}{Lemma}

\newtheorem{thm}{Theorem}[section]
\newtheorem{lem}[thm]{Lemma}

\theoremstyle{definition}

\theoremstyle{remark}

\newtheorem{ques}[thm]{Question}
\title{The Index of $\T_2^{\an}$ inside $\T_2$}
\author{Noah Taylor\footnote{The author is supported in part by NSF Grant DMS-$1701703$.}}
\date{February 28, 2021}
\begin{document}
\maketitle
\section{Introduction}\label{intro}
Let $N$ be a prime number and let $S_2(\Gamma_0(N),\Z)$ denote the modular forms of weight $2$ and level $\Gamma_0(N)$ with integer coefficients, and for any other ring $R$, we denote $S_2(\Gamma_0(N),R)=S_2(\Gamma_0(N),\Z)\tensor R$.  If $R$ is a characteristic $p$ ring, we define $S_2(\Gamma_0(N),R)^{\text{Katz}}$ to be the $R$-module of Katz forms as defined in \cite[Section 1.2]{MR0447119}, and define similar notation for the spaces of weight $1$ forms.  For $N\nmid n$, let $T_n$ denote the $n$th Hecke operator inside $\End(S_2(\Gamma_0(N),\overline{\Z}))$, and let $U_N$ denote the $N$th Hecke operator.  We let $\T^{\an}$ denote $\Z[T_3, T_5, \ldots]$, the algebra generated by $T_n$ for $(2N, n)=1$, and we denote $\T^{\an}[T_2, U_N]$ by $\T$.  The goal of this paper is to compute the index of $\T^{\an}$ inside $\T$.  Specifically, we prove the following theorem in sections \ref{integrality} and \ref{dimcount}:\begin{thm}\label{mainthm}The quotient $\T/\T^{\an}$ is purely $2$-torsion, and \[\dim_{\F_2}\T/\T^{\an}=\dim_{\F_2}S_1(\Gamma_0(N), \F_2)^{\Katz}.\]  In other words, if $c=\dim_{\F_2}S_1(\Gamma_0(N), \F_2)^{\Katz}$ is the dimension of the weight $1$ level $\Gamma_0(N)$ Katz forms over $\F_2$, then the index of $\T^{\an}$ in $\T$ is equal to $2^c$.\end{thm}

The setup of the paper is as follows.  In section \ref{prelims}, we introduce some facts from the literature about modular forms and establish a duality theorem between modular forms and Hecke algebras.  In section \ref{integrality} we prove the first half of the theorem, that $\T^{\an}$ contains $2\T$ as submodules of $\T$, so the quotient $\T/\T^{\an}$ is purely $2$-torsion.  Then in section \ref{dimcount} we use a theorem of Katz to relate the extra elements of $\T$ to weight $1$ modular forms using the duality, and finally establish the equality of Theorem \ref{mainthm} between dimensions.  In section \ref{examples} we conclude with some examples, and some theorems and conjectures we propose based on the work of Cohen-Lenstra and Bhargava. 

\section{Preliminaries}\label{prelims}
\subsection{From $\Z$ to $\Z_2$}
We start by proving that $U_N\in\T^{\an}$, thereby reducing our work to considering $\T^{\an}\subseteq\T^{\an}[T_2]$.

\begin{thm}\label{UNgood}$U_N\in\T^{\an}$.\end{thm}
\begin{proof}
It is enough to check that $U_N\in\T^{\an}\tensor\Z_p$ for every $p$: if $\T^{\an}$ and $\T^{\an}[U_N]$ have different ranks as $\Z$-modules, then the $\Z_p$-ranks of $\T^{\an}\tensor\Z_p$ and $\T^{\an}[U_N]\tensor\Z_p=\T^{\an}\tensor\Z_p[U_N]$ are also different for every $p$, contradiction.  On the other hand, if $\text{rank}(\T^{\an})=\text{rank}(\T^{\an}[U_N])$, then the quotient $\T^{\an}[U_N]/\T^{\an}$ is finite.  If it's nontrivial, then for any prime $p$ dividing its order, there is a surjective map $(\T^{\an}[U_N]\tensor\Z_p)/(\T^{\an}\tensor\Z_p)\twoheadrightarrow(\T^{\an}[U_N]/\T^{\an})\tensor\Z_p$ with nontrivial image.  So for this $p$, $\T^{\an}[U_N]\tensor\Z_p\neq\T^{\an}\tensor\Z_p$.  Therefore, we will only check whether $\T^{\an}\tensor\Z_p$ contains $U_N$.  Further, as $\T^{\an}\tensor\Z_p$ is a complete semi-local ring, it splits as a direct sum of its completions at maximal ideals, so it's further enough to check that $U_N$ is in $\T^{\an}_{\mathfrak{m}}$ for the completion $\T^{\an}_{\mathfrak{m}}$ at each maximal ideal $\mathfrak{m}$.

In a previous paper, we proved that $U_N\in\T^{\an}\tensor\Z_2=\T^{\an}_2$ \cite[Lemma 5.1]{NoahT}, so the statement is true for all maximal ideals over $2$.  So let $\ell$ be an odd prime, $\mathfrak{m}$ be a maximal ideal of $\T^{\an}$ over $\ell$, and $\mathfrak{a}$ be a maximal ideal of $\T$ containing $\mathfrak{m}$.

Let $\T_{\mathfrak{a}}$ be the completion of $\T$ with respect to $\mathfrak{a}$, and let $A$ be the integral closure of $\T_{\mathfrak{a}}$ over $\Z_{\ell}$, which can be written as $A=\displaystyle\oplus_i\mathcal{O}_i$ for $\mathcal{O}_i$ finite extensions of $\Z_{\ell}$.  The maps \[\pi_i:\T\rightarrow\T_{\mathfrak{a}}\rightarrow A\rightarrow\mathcal{O}_i\] produce conjugacy classes of eigenforms with coefficients in $\mathcal{O}_i$, with the coefficient $a_{i,j}$ of $q^j$ equal to $\pi_i(T_j)$ if $(j,N)=1$, or $\pi_i(U_j)$ if $N|j$.  These are newforms as $N$ is prime, and there are no weight $2$ level $1$ forms.  By Eichler-Deligne-Shimura-Serre there are representations $\rho_i:G_{\Q}\rightarrow\GL_2(\mathcal{O}_i)$, unramified away from $\ell N$, so that $\Tr(\rho_i(\Frob_{\ell}))=a_{i,p}$ for all primes $p\nmid \ell N$.

\cite[Theorem 3.1(e)]{MR1605752} describes the shape of the local-at-$N$ representation:\[\rho_i|_{G_{\Q_N}}=\begin{pmatrix}\epsilon\chi&*\\ 0&\chi\end{pmatrix}\]where $\chi$ is the unramified representation taking $\Frob_N$ to $a_{i,N}$ and $\epsilon$ is the $N$-adic cyclotomic character.  Additionally, $\det\rho_i=\epsilon$, so $\chi^2$ is identically $1$ and $a_{i,N}$ is equal to $1$ or $-1$ for each $i$.  We show that $a_{i,N}$ is equal among all $i$ over all $\mathfrak{a}$ containing $\mathfrak{m}$, so that the image of $U_N$ in $\T_{\mathfrak{a}}$ is constantly $1$ or $-1$ over all $\mathfrak{a}$, and hence, in $\T_{\mathfrak{m}}=\oplus_{\mathfrak{m}\subseteq\mathfrak{a}}\T_{\mathfrak{a}}$, is inside $\T^{\an}_{\mathfrak{m}}$.

By the Chebotarev density theorem, a representation is determined up to semisimplification and conjugation by its trace on the Frobenius elements of unramified primes.  The $\rho_i(\Frob_p)$ have trace equal to $a_{i,p}$, which is the image of $T_p$ under $\pi_i$.  Because $\mathfrak{m}$ is contained in $\mathfrak{a}$ for all $\mathfrak{a}$, the image of $T_p$ under reduction of $\T^{\an}\bmod\mathfrak{m}$ is the same as the reduction of $a_{i,p}\bmod\mathfrak{a}$.  Therefore, the semisimplifications of the reductions of $\rho_i$ over all $i$ and all $\mathfrak{a}$ are all isomorphic.  But we can deduce the value of $a_{i,N}$ from the reduction of $\rho_i\bmod\mathfrak{a}$, because $\rho_i|_{G_{\Q_N}}$ has an unramified quotient and a ramified subspace, and the same is true for the reduction $\bmod\mathfrak{a}$ as $\ell\neq2$.  So the image of the Frobenius on the unramified quotient is either $1$ or $-1$ for one (and hence every) $\rho_i$, and therefore $a_{i,N}$ does not depend on $i$ or $\mathfrak{a}$, only on $\mathfrak{m}$.  So $U_N$ lies in $\T^{\an}_{\mathfrak{m}}$ for all $\mathfrak{m}$, and we're done.\end{proof}

We can now reduce from forms over $\overline{\Z}$ to forms over $\overline{\Z}_2$.  To do this, we first recall \cite[Lemma, pp. 491]{MR1333035} which says that if $\T^1$ is the Hecke algebra over $\Z$ corresponding to level $\Gamma_1(N)$ forms and $\T^2$ is the subalgebra of operators relatively prime to $2$, that $\T^2$ has $2$-power index in $\T^1$.  As the algebras $\T$ and $\T^{\an}=\T^{\an}[U_N]$ are quotients of $\T^1$ and $\T^2$ of this lemma, the same is true for $\T$ and $\T^{\an}$.  (Alternatively, with a similar argument to the proof of Theorem \ref{UNgood}, we can check that $T_2$ is contained in all completions at maximal ideals of $\T^{\an}\left[\frac{1}{2}\right]$.  This is true as $2$ is unramified in, and $T_2$ is a trace of, the modular representations over primes other than $2$, so Chebotarev and completeness of $\T^{\an}_{\mathfrak{m}}$ show that $T_2\in\T^{\an}_{\mathfrak{m}}$.)  So we can calculate the index of $\T^{\an}\tensor\Z_2$ inside $\T\tensor\Z_2$, and by abuse of notation begin to call these $\T^{\an}$ and $\T$ instead.  We know that $\T$ and $\T^{\an}$ are semi-local rings, and as such, they can be written as a direct sum of their completions: \[\T=\bigoplus_{\mathfrak{a}\subset\T}\T_{\mathfrak{a}},\qquad\text{and}\qquad\T^{\an}=\bigoplus_{\mathfrak{m}\subset\T^{\an}}\T^{\an}_{\mathfrak{m}}.\]

Additionally, because the $\Z_2$-ranks of $\T$ and $\T^{\an}$ are equal, $T_2\in\T\tensor\Q_2=\T^{\an}\tensor\Q_2=\T^{\an}\left[\frac{1}{2}\right]$, and hence maps $\T^{\an}\rightarrow K$ where $K$ is a finite extension of $\Q_2$ can be uniquely extended to maps $\T\rightarrow K$.  This means that modular forms are rigid in characteristic $0$: we can determine the image of $T_2$ from the image of the remaining operators, and hence from any modular representation $\rho_f:G_{\Q}\rightarrow\GL_2(K)$ we may determine the entire form $f$.  We say that $\rho$ is ordinary if the restriction $\rho|_{D_2}$ of $\rho$ to the decomposition group at $2$ is reducible, and we say that an eigenform is ordinary if $a_2$ is a unit mod 2.  The next theorem describes the shape of $\rho_f$ at $2$:
\begin{thm}[{\cite[Theorem 2]{Wiles1988}}]\label{wiles}If $f$ is an ordinary $2$-adic form, then $\rho_f|_{D_2}$, the restriction of $\rho_f$ to the decomposition group at a prime above $2$, is of the shape\[\rho_f|_{D_2}\sim\begin{pmatrix}\chi\lambda^{-1}&*\\ 0&\lambda\end{pmatrix}\]for $\lambda$ the unramified character $G_{\Q_2}\rightarrow\overline{\Z}_2^{\times}$ taking $\Frob_2$ to the unit root of $X^2-a_2X+2$, and $\chi$ is the $2$-adic cyclotomic character.\end{thm}

\subsection{A Duality Theorem}
In this section, we will compute the Pontryagin dual of one of the summands in $\T$ with the following lemma.  Let $\mathfrak{a}$ be any maximal ideal of $\T$ and let \[S_2(\Gamma_0(N),\Z_2)_{\mathfrak{a}}=e\cdot S_2(\Gamma_0(N),\Z_2)\] where $e$ is the projector $\T\rightarrow\T_{\mathfrak{a}}$.
\begin{lem}\label{duality}The Pontryagin dual of $\T_{\mathfrak{a}}$ is $M=\displaystyle\lim_{\longrightarrow}S_2(\Gamma_0(N),\Z_2)_{\mathfrak{a}}/(2^n)$ where the transition maps are multiplication by $2$.\end{lem}

\begin{proof}First, we note that $\T_{\mathfrak{a}}$ acts on $M$ because $\T_{\mathfrak{a}}$ acts compatibly on each level.  If any element $T\in\T_{\mathfrak{a}}$ acts trivially on $M$, then on any given modular form in $S_2(\Gamma_0(N),\Z_2)_{\mathfrak{a}}$, it acts by arbitrarily high powers of $2$, and hence acts as $0$.  Then $T$ acts trivially on the rest of $S_2(\Gamma_0(N),\Z_2)$, so $T$ is the $0$ endomorphism.  Therefore, $M$ is a faithful $\T_{\mathfrak{a}}$-module.

We also know that $M[\mathfrak{a}]$, the elements of $M$ killed by all of $\mathfrak{a}$, is a subspace of $S_2(\Gamma_0(N),\Z_2)_{\mathfrak{a}}/(2)=S_2(\Gamma_0(N), \F_2)_{\mathfrak{a}}$.  It is a vector space over $\T/\mathfrak{a}$, although through the action of $\T$, not by multiplication on the coefficients.  We explain why it's a $1$-dimensional $\T/\mathfrak{a}$-vector space.  The map \[S_2(\Gamma_0(N), \F_2)\rightarrow\Hom(\T, \F_2),\quad f\mapsto\phi_f:T_n\rightarrow a_n\] is injective by the $q$-expansion principle.  The forms killed by $\mathfrak{a}$ must correspond to maps factoring through $\T/\mathfrak{a}$, so the space of forms is at most the dimension of $\Hom(\T/\mathfrak{a},\F_2)=\dim_{\F_2}\T/\mathfrak{a}$.  So the dimension as a $\T/\mathfrak{a}$-vector space is at most $1$.

On the other hand, there is at least $1$ form in $M[\mathfrak{a}]$, because we may take the form $T_1q+T_2q^2+T_3q^3+\ldots\in S_2(\Gamma_0(N),\T/\mathfrak{a})$ and consider its image under the trace map $\T/\mathfrak{a}\rightarrow\F_2$.  This is nonzero because the trace map is nondegenerate, and because the Hecke operators generate $\T$ additively.  This is in the kernel of $\mathfrak{a}$ because the trace of a form is just the sum of its conjugates, and for any expression in $\mathfrak{a}$ in terms of the Hecke operators with coefficients in $\F_2$, because its application to the original form is $0$ by definition, its application to any of the form's conjugates must also be $0$ (because the Hecke operators act $\F_2$-linearly on a form's coefficients and hence commute with Galois conjugation), and so too must its application to the sum.  Because the trace form has coefficients in $\F_2$, we've found a nontrivial form in $M[\mathfrak{a}]$, and this must be dimension $1$ as required.

We consider the Pontryagin dual of $M$: as $M$ is a $\Z_2$-module, the image of any map $M\rightarrow\Q/\Z$ must land in $\Q_2/\Z_2$.  So let $M^{\vee}=\Hom_{\Z_2}(M,\Q_2/\Z_2)$.  We endow this with a $\T_{\mathfrak{a}}$-module structure by letting $(T\phi)(f)=\phi(Tf)$.  Because $S_2(\Gamma_0(N),\Z_2)_{\mathfrak{a}}\simeq\Z_2^k$ for some $k$ because it is torsion free, $M\simeq(\Q_2/\Z_2)^k$ as a $\Z_2$ module.  So if $\phi(f)=0$ for all $\phi\in M^{\vee}$, we know that $f=0$.  If $T\phi=0$ for all $\phi$, then $\phi(Tf)=0$ for all $\phi$ and $f$, and so $Tf=0$ for all $f$, and $T=0$.  So $M^{\vee}$ is also a faithful $\T_{\mathfrak{a}}$-module.

Further, $\T_{\mathfrak{a}}$ injects into $M^{\vee}$: we can rewrite \[M=\displaystyle\lim_{\longrightarrow}\frac{1}{2^n}S_2(\Gamma_0(N),\Z_2)_\mathfrak{a}/S_2(\Gamma_0(N),\Z_2)_{\mathfrak{a}}\] where the transition maps are inclusion.  Then the $\T_{\mathfrak{a}}\times M\rightarrow \Q_2/\Z_2$ as $(T, f)\rightarrow a_1(Tf)$ defines the injection.  By Nakayama's lemma and the duality of $M[\mathfrak{a}]$ and $M^{\vee}/\mathfrak{a}$, the minimal number of generators of $M^{\vee}$ as a $\T_{\mathfrak{a}}$-module is $1$.  So we've proven that $M^{\vee}\simeq\T_{\mathfrak{a}}$.\end{proof}

We may use Pontryagin Duality to find that the dual to $T_{\mathfrak{a}}/2=M^{\vee}/2$ is $M[2]$, which is exactly $S_2(\Gamma_0(N),\Z_2)_{\mathfrak{a}}/(2)=S_2(\Gamma_0(N),\F_2)_{\mathfrak{a}}$.  Thus we obtain a perfect pairing \[T_{\mathfrak{a}}/2\times S_2(\Gamma_0(N),\F_2)_{\mathfrak{a}}\rightarrow \F_2, \qquad (T, f)\rightarrow a_1(Tf).\]We may sum these pairings over all $\mathfrak{a}$, because Hecke operators and forms with incompatible maximal ideals annihilate each other.  Therefore we obtain a perfect pairing $\T/2\times S_2(\Gamma_0(N),\F_2)\rightarrow \F_2$.

\section{$2T_2$ is integral}\label{integrality}
In this section we prove the following lemma:
\begin{lem}\label{doubleT}For any element $T\in\T$, the element $2T\in\T$ lies inside $\T^{\an}$.\end{lem}
First we prove a lemma describing the image of the representation corresponding to a non-Eisenstein ideal.
\begin{lem}\label{carayollike}Suppose $\mathfrak{m}$ does not contain the Eisenstein ideal.  Then there is a representation \[\rho:G_{\Q}\rightarrow\GL_2(\T^{\an}_{\mathfrak{m}}).\] that is unramified outside $2N$, and which satisfies $\Tr(\rho(\Frob_{\ell}))=T_{\ell}$ for $\ell\nmid 2N$.\end{lem}
\begin{proof}
Let $A=\T^{\an}_{\mathfrak{m}}$ and $A'$ is its integral closure over $\Z_2$, which can be written as the product $\prod_i \mathcal{O}_i$ of a collection of integer rings.  We know that there exist representations $\rho'_i: G_{\Q}\rightarrow \prod_i \GL_2(\mathcal{O}_i)$, by Eichler-Shimura-Deligne-Serre.  The image is $\GL_2(\mathcal{O}_i)$, because $G_{\Q}$ is compact, and we may choose an invariant lattice on which it acts.  These $\rho'_i$ combine to give a representation \[\rho'=\prod_i \rho'_i: G_{\Q}\rightarrow \prod_i \GL_2(\mathcal{O}_i).\]

We know that the traces of the representations at $\Frob_{\ell}$ are the images of $T_{\ell}$ for all $\ell\nmid pN$, so the trace of $\rho'$ by Chebotarev Density always lands in $\T^{\an}_{\mathfrak{m}}$.  We assumed $\mathfrak{m}$ did not contain the Eisenstein ideal, so we know that each $\rho'_i$, and therefore the full $\rho'$, is residually irreducible.  By \cite[Theorem 2]{MR1279611} we find that $\rho'$ is similar to a representation \[\rho:G_{\Q}\rightarrow\GL_2(\T^{\an}_{\mathfrak{m}}).\]\end{proof}
To prove Lemma \ref{doubleT}, we look at the three different possible cases and deduce that the projection of $2T_2$ to $\T_{\mathfrak{a}}$ lies in $\T^{\an}_{\mathfrak{m}}$ for each $\mathfrak{m}\subseteq\mathfrak{a}$.  Further, we prove that $T_2^2$ lies in $\T^{\an}_{\mathfrak{m}}\cdot T_2+\T^{\an}_{\mathfrak{m}}$, so that any $T\in\T$, being an element in $\T^{\an}[T_2]$, lies in $\T^{\an}_{\mathfrak{m}}\cdot T_2+\T^{\an}_{\mathfrak{m}}$ also, and hence is half of an element in $\T^{\an}_{\mathfrak{m}}$.
\subsection{$\overline{\rho}$ ordinary irreducible}
We first assume that the residual representation $G_{\Q}\rightarrow\GL_2(\T^{\an}_{\mathfrak{m}}/\mathfrak{m})$ is irreducible but the local residual representation at $2$ is reducible.  We will show that $2T_2$, as an element of $\T^{\an}_{\mathfrak{m}}[T_2]$, actually lies in $\T^{\an}_{\mathfrak{m}}$.  This will be done by proving it is in the ring generated over $\Z_2$ by the traces of $\rho$.  Equivalently, we will look at the traces of $\rho\otimes_{\Z_2}\Q_2$.  This breaks the representation into a direct sum $\bigoplus_i\rho'_i\tensor\Q_2: G_{\Q}\rightarrow\prod_i\GL_2(E_i)$.  Each of the $\rho'_i$ themselves have the same residual representation which is reducible when restricted to the decomposition group, so all these representations are ordinary.

Looking at a given $\rho'_i$, we may apply Theorem \ref{wiles} to it to obtain a shape of $\rho'_i|_{D_2}$.  In particular, the trace of an element $\rho(g)$ is equal to $\chi(g)\lambda^{-1}(g)+\lambda(g)$ with $\lambda$ the unramified character whose image of Frobenius is the unit root of $X^2-T_2X+2$, and $\chi$ is the cyclotomic character.  If $\alpha$ denotes the unit root of $x^2-a_{2, i}x+2=0$, then letting $g$ be an element of $\Gal(\Q_2^{\text{ab}}/\Q_2)$ which both is a lift of Frobenius and acts trivially on the $2$-power roots of unity (so $\chi(g)=1$), then we know $\Tr(g)=\alpha+\alpha^{-1}$.  If we let $h$ be a lift of Frobenius with $\chi(h)=-1$, we find that $\Tr(h)=\alpha-\alpha^{-1}$.  And by definition, we know $\alpha+\frac{2}{\alpha}=a_{2,i}$, so $2a_{2, i}=2\alpha+4\alpha^{-1}=3\Tr(g)-\Tr(h)$.

We now look at the product of representations.  The elements $g$ and $h$ were independent of the coefficient field, so we know that the element of $\T^{\an}_{\mathfrak{m}}\tensor\Q_2$ that is $2a_{2, i}$ in each coordinate, namely $2T_2\tensor1$, is equal to $3\Tr(g)-\Tr(h)$.  So $2T_2$ is in the ring generated by the traces of elements, and thus in $\T^{\an}_{\mathfrak{m}}$.

Similarly, we can prove that $T_2^2$ is in $\T^{\an}_{\mathfrak{m}}+T_2\cdot \T^{\an}_{\mathfrak{m}}$: in each coordinate, we can calculate that \[a_{2, i}^2=\Tr(g)a_{2, i}+(\Tr(gh)-\Tr(g^2)-1).\]  So in $\T^{\an}_{\mathfrak{m}}[T_2]$, we find that $T_2^2=\Tr(g)T_2+(\Tr(gh)-\Tr(g^2)-1)$.  So $T_2^2\subseteq\T^{\an}_{\mathfrak{m}}+T_2\cdot \T^{\an}_{\mathfrak{m}}$, and therefore so is every power of $T_2$.  So we know that $2\T^{\an}_{\mathfrak{m}}[T_2]\subseteq\T^{\an}_{\mathfrak{m}}$, and the $\T^{\an}_{\mathfrak{m}}$-module quotient $\T^{\an}_{\mathfrak{m}}[T_2]/\T^{\an}_{\mathfrak{m}}$ is an $\F_2$ vector space.  In section \ref{dimcount} we will calculate its dimension.

\subsection{$\overline{\rho}$ reducible}\label{reducible}
We now suppose $\T^{\an}_{\mathfrak{m}}$ corresponds to a reducible residual representation, so that $\mathfrak{m}$ is the Eisenstein ideal generated by $2$ and $T_{\ell}$ for $\ell\nmid N$ (including $\ell=2$).  We claim that $T_2$ is already in $\T^{\an}_{\mathfrak{m}}$.  This is because by \cite[Proposition 17.1]{MR488287}, the Eisenstein ideal of the full Hecke algebra is generated by $1+\ell-T_{\ell}$ for any good prime.  So by completeness, $T_2-3$ and therefore $T_2$ can be written as a power series in $T_\ell-\ell-1$.

\subsection{$\overline{\rho}$ non-ordinary}\label{supersingular}
We now assume that the residual local representation at $2$ is irreducible, or equivalently that in $\T_{\mathfrak{a}}$, $T_2$ is not a unit, where $\mathfrak{a}$ is some ideal of $\T$ above $\mathfrak{m}$ corresponding to $\rho$.  We claim that $T_2$ is already in $\T^{\an}_{\mathfrak{m}}$, so that $\mathfrak{a}=\mathfrak{m}$ is actually unique, and the index is $1$.

\begin{thm}\label{nonord}If $\rho$ is non-ordinary with corresponding map $\T^{\an}\rightarrow\F$ with maximal ideal $\mathfrak{m}$, then for any $\mathfrak{a}\subseteq\T$ containing $\mathfrak{m}$, $T_2\in\T_{\mathfrak{a}}$ is already contained in the image of $\T^{\an}_{\mathfrak{m}}$.\end{thm}
\begin{proof}The $\T^{\an}_{\mathfrak{m}}$-module $\T^{\an}_{\mathfrak{m}}[T_2]$ requires the same generators as the $\T^{\an}/\mathfrak{m}$-vector space $\T/\mathfrak{m}\T$ by Nakayama's Lemma, so it's enough to prove that $\T/\mathfrak{m}\T$ is one-dimensional over $\T^{\an}/\mathfrak{m}$.  If it's not, then all of $\T^{\an}/\mathfrak{m}$ and $T_2$ are independent over $\F_2$, so there is a homomorphism $\phi\in\Hom(\T/\mathfrak{m}\T,\F_2)$ sending all of $\T^{\an}/\mathfrak{m}$ to $0$, and $T_2$ to $1$.  Recalling the perfect pairing after Lemma \ref{duality}, we find a nonzero modular form $g\in S_2(\Gamma_0(N),\F_2)[\mathfrak{m}]$ with all odd coefficients equal to $0$.

By part (3) of the main result of \cite{MR0463169}, we know that there is some nonzero form $f\in S_1(\Gamma_0(N), \F_2)^{\Katz}$ with $f^2=g$.  (Here, we're considering weight $1$ Katz forms, and so the weight $2$ forms we construct may be Katz forms as well.  So if necessary we enlarge the spaces we're considering, but it doesn't affect the conclusion.)  As forms with coefficients in $\F_2$ commute with the Frobenius endomorphism, $f(q^2)$ has the same $q$-expansion as $g$.  If $\T^1$ and $\T^{1,\an}$ are the weight $1$ Hecke algebras, it is quick to check that the corresponding Hecke actions on $q$-expansions of $\T^{1,\an}$ are identical to those of $\T^{\an}$.  Therefore $f\in S_1(\Gamma_0(N), \F_2)^{\Katz}[\mathfrak{m}]$.  Further, we know that $f$ is alone in this space, by part (2) of \cite{MR0463169}: any other form in $S_1(\Gamma_0(N), \F_2)^{\Katz}[\mathfrak{m}]$ has the same odd coefficients, so the difference between it and $f$ has only even-power coefficients, and hence must be $0$ by Katz's theorem.  So $f$ is also an eigenform for $T_2$ in weight $1$, say with eigenvalue $b_2$.

So we've discovered that $S_2(\Gamma_0(N), \F_2)^{\Katz}[\mathfrak{m}]$ is at most $2$ dimensional, spanned by $Vf$ and $Af$.  Here, $V$ acts as $V\left(\sum_{n=1}^{\infty}a_nq^n\right)=\sum_{n=1}^{\infty}a_nq^{2n}$ on power series, so that $Vf=g$, and can either be a weight-doubling operator, as used in \cite{MR0463169}, or a level-doubling operator.  Additionally, $Af$ is the multiplication of $f$ with the Hasse invariant $A$, which preserves $q$-expansions.  We can hence calculate the action of $T_2$ on this space: we know that $T_2$ acts in weight $2$ via $U+2V$, where $U\left(\sum_{n=1}^{\infty}a_nq^n\right)=\sum_{n=1}^{\infty}a_{2n}q^n$, and in weight $1$ as $U+\langle2\rangle V$ with $\langle2\rangle$ the diamond operator, which is identically $1$ on mod 2 forms.  Further, we can compute that $UVf=Af$, as $V$ doubles each exponent and $U$ halves it.

So we find\begin{align*}T_2(Vf)&=UVf=Af\\ T_2(Af)&=U(Af)=A(Uf)=A(T_2f-\langle2\rangle Vf)=A(b_2f)-\langle2\rangle Vf\end{align*}and the matrix for the $T_2$ action is $\begin{pmatrix}b_2&-\langle2\rangle\\ 1&0\end{pmatrix}$.  (In these computations, the distinction between the level-raising $V$ and the weight-raising $V$ has been blurred, because on $q$-expansions they are equal; we view both lines as equalities of weight $2$ level $\Gamma_0(N)$ forms.)  As $\langle2\rangle$ is trivial, the determinant of this matrix is $1$, so $T_2$ is invertible.  This is impossible because the form was non-ordinary.  So there cannot be such a form $g$, and $\T^{\an}_{\mathfrak{m}}[T_2]$ requires only one generator as a $\T^{\an}_{\mathfrak{m}}$-module, as required.
\end{proof}
\section{Dimension of $\T/\T^{\an}$}\label{dimcount}
In this section we prove the second half of Theorem \ref{mainthm}.  It is enough to look locally, so we will localize at a maximal ideal $\mathfrak{m}$ of $\T^{\an}$.  Because completion at only ordinary non-Eisenstein ideals have $T_2$ not immediately in $\T^{\an}_{\mathfrak{m}}$, we assume that $\mathfrak{m}$ is such an ideal.

\subsection{Relating $\T/\T^{\an}$ to $S_2$}

We first recall the perfect pairing $S_2(\Gamma_0(N), \F_2)\times \T/2\rightarrow \F_2$, given by $(f, T)\rightarrow a_1(Tf)$.  While proving this, we proved perfect pairings $S_2(\Gamma_0(N), \F_2)_{\mathfrak{a}}\times\T_{\mathfrak{a}}/2\rightarrow\F_2$, and we now combine all $\mathfrak{a}$ that contain $\mathfrak{m}$, to get a perfect pairing $S_2(\Gamma_0(N), \F_2)_{\mathfrak{m}}\times\T_{\mathfrak{m}}/2\rightarrow\F_2$ where we denote $\T_{\mathfrak{m}}$ as the localization of $\T$ at the (not necessarily maximal) ideal $\mathfrak{m}\T$, and $S_2(\Gamma_0(N), \F_2)_{\mathfrak{m}}=e\cdot S_2(\Gamma_0(N), \F_2)$ for $e$ the projection from $\T$ to $\T_{\mathfrak{m}}$.  Considering the subspace of forms killed by $A\theta$, the operator defined in \cite{MR0463169} which acts as $q\frac{d}{dq}$ on $q$-expansions and raises the weight by $3$, it's clear that the entirety of $\T^{\an}_{\mathfrak{m}}$ annihilates it under the pairing, and we wish to prove that this is the full annihilator.  For ease of notation, let us write $V=\T_{\mathfrak{m}}/2\T_{\mathfrak{m}}$, $W=S_2(\Gamma_0(N), \F_2)_{\mathfrak{m}}$, and $V'=\T_{\mathfrak{m}}^{\an}/2\T_{\mathfrak{m}}$.

\begin{lem}\label{annihilators}$S_2(\Gamma_0(N), \F_2)_{\mathfrak{m}}\cap\Ker A\theta$ and $\T_{\mathfrak{m}}^{\an}/2\T_{\mathfrak{m}}$ are mutual annihilators in this perfect pairing.\end{lem}
\begin{proof}We've seen that they annihilate each other.  Now suppose $f=\sum_{i=1}^{\infty}a_iq^i\in W$ is annihilated by all of $V'$.  By the usual formula for the Hecke action on $q$-expansions, the coefficient of $q^1$ in $T_nf$ is $a_n$, so $a_n=0$ for all odd $n$.  Therefore $f\in S_2(\Gamma_0(N), \F_2)_{\mathfrak{m}}\cap\Ker A\theta$, and we can call this space $\Ann(V')$.  This is enough to show they are mutual annihilators by dimension count, but we'll prove the other direction as well.

The space $W/\Ann(V')$ is represented by sequences of odd-power coefficients that appear in forms in $W$.  We first prove that the map $V'\rightarrow\Hom(W/\Ann(V'),\F_2)$ induced by the pairing is surjective.  Given a map $\varphi\in\Hom(W/\Ann(V'),\F_2)$ whose input is sequences of odd-power coefficients, we can define a map $\varphi'$ in the double dual of $V'$ taking maps \[\chi: V'\rightarrow \F_2\text{ to }\varphi(\chi(T_1), \chi(T_3), \chi(T_5), \ldots).\]  This is the definition of $\varphi'$ when $(\chi(T_1), \chi(T_3), \ldots)$ appears as the odd-power coefficients of a form.  And then if we've not defined $\varphi'$ on all of the dual of $V'$, we can just extend it any way we want.  But because $V'$ is finite dimensional, this $\varphi'$ determines an element $T_{\varphi}\in V'$ for which \[\chi(T_{\varphi})=\varphi'(\chi)=\varphi(\chi(T_1), \chi(T_3), \ldots).\]  Then because any sequence of coefficients $(a_1, a_3, \ldots)$ is given by a character $\chi_{(a_i)}: T_n\rightarrow a_n$ (the restriction of such a $\chi$ from $\T_{\mathfrak{a}}$, for example), the pairing truly does send $T_{\varphi}$ to $\varphi$.

Now given $T$ that sends all of $\Ann(V')$ to $0$, $Tf$ must only depend on the odd coefficients of $f$.  But then $\varphi: f\rightarrow a_1(Tf)$ is an element of $\Hom(W/\Ann(V'),\F_2)$.  So by surjectivity there is some element $T'$ of $V'$ with $a_1(Tf)=\varphi(f)=a_1(T'f)$ for all $f\in W/\Ann(V')$.  Then $a_1((T-T')f)$ is $0$ for all $f$ either in $\Ann(V')$ or a lift of an element of $W/\Ann(V')$, and so in all of $W$.  Because the pairing is perfect, $T=T'\in V'$ as we needed.\end{proof}

Now that we know these are mutual annihilators, we obtain an isomorphism \[V/V'\rightarrow\Hom(\Ann(V'),\F_2),\] and taking dimensions and reinterpreting, we've proven that \[\dim\T_{\mathfrak{m}}/\T^{\an}_{\mathfrak{m}}=\dim S_2(\Gamma_0(N), \F_2)_{\mathfrak{m}}\cap\Ker A\theta.\]So we have proven the following.

\begin{lem}\label{prelimind}The index of $\T^{\an}_{\mathfrak{m}}$ in $\T_{\mathfrak{m}}$ equals $2$ raised to the dimension of $S_2(\Gamma_0(N), \F_2)_{\mathfrak{m}}\cap\Ker A\theta.$\end{lem}

\subsection{Lifting from weight $1$ to weight $2$}
Now we use the main theorem of \cite{MR0463169} to find a subspace of $S_1(\Gamma_0(N), \F_2)^{\Katz}$ that maps under $V$ to $S_2(\Gamma_0(N), \F_2)_{\mathfrak{m}}\cap\Ker A\theta$.  As in Section \ref{supersingular}, we have $\T^{\an}$-equivariance, and so the maximal ideal $\mathfrak{m}$ has an exact analogue in $\T^{1,\an}$ and we land in the subspace $S_1(\Gamma_0(N), \F_2)_{\mathfrak{m}}^{\Katz}$.  We may not obtain the whole subspace because, while $Vf$ is in the kernel of $A\theta$ for all $f\in S_1(\Gamma_0(N), \F_2)_{\mathfrak{m}}^{\Katz}$, we don't know that it's a form that is the reduction of a $\Z_2$ form, which is what $\T^{\an}_{\mathfrak{m}}$ parametrizes.  In this section we will prove that the space of Katz forms of weight $2$ actually are all standard forms.

The first case is $N\equiv 3\mod 4$, which was taken care of Edixhoven:
\begin{thm}[{\cite[Theorem 5.6]{MR2195943}}]\label{edixhoven} Let $N\geq 5$ be odd and divisible by a prime number $q\equiv-1$ modulo $4$ (hence the stabilizers of the group $\Gamma_0(N)/\{1,-1\}$ acting on the upper half plane have odd order).  Then $S_2(\Gamma_0(N), \F_2)^{\Katz}$ and $\F_2\tensor S_2(\Gamma_0(N), \Z)$ are equal, and the localizations at non-Eisenstein maximal ideals of the algebras of endomorphisms of $S_2(\Gamma_0(N), \F_2)^{\Katz}$ and $H^1_{\textup{par}}(\Gamma_0(N),\F_2)$ generated by all $T_n$ ($n\geq 1$) coincide: both are equal to that of $S_2(\Gamma_0(N), \Z)$ tensored with $\F_2$.\end{thm}

So for primes $N\equiv 3\mod 4$, we've proven the equality in Theorem \ref{mainthm}.  For the remainder of this section we therefore assume $N\equiv 1\mod 4$.  Further, up until this point we've only worked with $\F_2$-forms, but we change coefficients to $\overline{\F}_2$ so that we can find eigenforms associated to each maximal ideal.  Theorem \ref{edixhoven} still applies as its proof in \cite{MR2195943} can be extended to all finite extensions of $\F_2$.

\begin{thm}\label{nokatz}There are no Katz forms that are not the reduction of a form in $S_2(\Gamma_0(N),\overline{\Z}_2)$.  That is, \[S_2(\Gamma_0(N),\overline{\F}_2)^{\Katz}=S_2(\Gamma_0(N),\overline{\F}_2).\]\end{thm}
\begin{proof}Let $\ell$ be an arbitrary prime that is $3\mod 4$, and we will look at $S_2(\Gamma_0(N\ell),\overline{\F}_2)^{\Katz}$.  We can apply Theorem \ref{edixhoven} to it and conclude that this space is exactly the characteristic $0$ forms tensored with $\overline{\F}_2$, so we may drop the Katz superscript.  Further, we know that all Katz forms of level $\Gamma_0(N)$ lie in this space.  So we just need to know there are no extra level $\Gamma_0(N)$ forms within this space.

As $\T^{\Katz}\tensor\overline{\F}_2$ can be broken into a direct sum of $\overline{\F}_2$-vector spaces on which the semi-simple action of each operator is by multiplication by a constant, $S_2(\Gamma_0(N), \overline{\F}_2)^{\Katz}$ can be written as a direct sum of generalized eigenspaces.  If we show every generalized eigenform in $S_2(\Gamma_0(N), \overline{\F}_2)^{\Katz}$ is the reduction of a modular form from $S_2(\Gamma_0(N), \overline{\Z}_2)$, then we're done.  So suppose $f$ is a generalized Katz eigenform for all $T_n$, including $T_2$.  Let the eigenvalue corresponding to $T_{\ell}$ equal $a_{\ell}$; we will prove that if $f\not\in S_2(\Gamma_0(N), \overline{\F}_2)$, then $a_{\ell}=0$.

There are two maps from $S_2(\Gamma_0(N),\overline{\F}_2)^{\Katz}$ to $S_2(\Gamma_0(N\ell),\overline{\F}_2)$: the plain embedding with equality on $q$-expansions, and the map $V_{\ell}$ sending $f(q)$ to $f(q^{\ell})$.  We know $T_{\ell}=U_{\ell}+\ell V_{\ell}$ on $q$-expansions, so we find that \[U_{\ell}(T_{\ell}-a_{\ell})=U_{\ell}(U_{\ell}+\ell V_{\ell}-a_{\ell})=U_{\ell}^2-a_{\ell}U_{\ell}+\ell U_{\ell}V_{\ell}=U_{\ell}^2-a_{\ell}U_{\ell}+\ell\]as operators from $S_2(\Gamma_0(N),\overline{\F}_2)^{\Katz}$ to $S_2(\Gamma_0(N\ell),\overline{\F}_2)$.  Then because $f$ is a generalized eigenform, we find \[0=(U_{\ell}^k(T_{\ell}-a_{\ell})^k)f=U_{\ell}^{k-1}(U_{\ell}^2-a_{\ell}U_{\ell}+\ell)(T_{\ell}-a_{\ell})^{k-1}f=\ldots=(U_{\ell}^2-a_{\ell}U_{\ell}+\ell)^kf.\]If we factor $X^2-a_{\ell}X+\ell$ as $(X-\alpha)(X-\beta)$ for some lift of $a_{\ell}$, we've proven that $(U_{\ell}-\alpha)(U_{\ell}-\beta)$ acts topologically nilpotently on any lift of $f$ (which exists by Theorem \ref{nokatz}).  This will eventually be used to prove that one of $\alpha$ or $\beta$, and hence both, reduce to $1$ mod the maximal ideal of $\overline{\Z}_2$.

\begin{lem}\label{nilpotent}For any characteristic $0$ newform $g$ of level $N\ell$, $U_{\ell}-1$ acts topologically nilpotently.\end{lem}\begin{proof}The eigenform $g$ gives us a representation $\rho:G_{\Q}\rightarrow\GL_2(\overline{\Q}_2)$.  The shape of this representation at the decomposition group at $\ell$ is given by \cite[Theorem 3.1(e)]{MR1605752}, as we recalled in the proof of Theorem \ref{UNgood}, which says that \[\rho|_{D_{\ell}}=\begin{pmatrix}\chi\varepsilon&*\\ 0&\chi\end{pmatrix}\] where $\chi$ is the unramified representation that sends $\Frob_{\ell}$ to the $U_{\ell}$-eigenvalue of $g$, and $\varepsilon$ is the $2$-adic cyclotomic character.  Because the determinant is the $2$-adic cyclotomic character as well, we know that $\chi^2=1$, so the $U_{\ell}$-eigenvalue of $g$ is $\pm 1$.  So $U_{\ell}-1$ is either $0$ or $-2$, which both act nilpotently.\end{proof}

If $\alpha-1$ and $\beta-1$ have valuation $0$, then $(U_{\ell}-\alpha)(U_{\ell}-\beta)$ will not act nilpotently on any linear combination of eigenforms which includes at least one newform, by Lemma \ref{nilpotent}.  As $(U_{\ell}-\alpha)(U_{\ell}-\beta)$ acts nilpotently on a lift of $f$, we know that this lift is a linear combinaton of only oldforms, and hence $f$ lifts to $S_2(\Gamma_0(N),\overline{\Z}_2)$.  Otherwise, one of $\alpha$ and $\beta$, and hence both, are $1$ mod the maximal ideal of $\overline{\Z}_2$, and so $\alpha+\beta\equiv0\equiv a_{\ell}$.

Therefore, we have proven that if $f$ is a generalized eigenform in $S_2(\Gamma_0(N),\overline{\F}_2)^{\Katz}$ that has no lift to characteristic $0$, then $a_{\ell}=0$ for any prime $\ell\equiv3\mod 4$, as our choice of $\ell$ was arbitrary.  Letting $g$ be a true eigenform in the same eigenspace as $f$, we obtain a representation $\overline{\rho}_g:G_{\Q}\rightarrow\GL_2(\overline{\F}_2)$ with $\Tr(\rho_g(\Frob_p))=a_p$.  We showed that $\overline{\rho}_g$ has trace $0$ on all $\Frob_{\ell}$, so it must be the induction of a character from $G_{\Q(i)}$ to $G_{\Q}$.  But such a representation is dihedral in the terminology of \cite{Kedlaya2019}, and \cite[Theorem 12(1)]{Kedlaya2019} proves that it's impossible for a dihedral representation on $G_{\Q(i)}$ to give rise to a form of level $\Gamma_0(N)$.  So there can be no Katz eigenforms of level $\Gamma_0(N)$ that don't lift, and hence no generalized eigenforms and therefore no forms at all.\end{proof}

From this, we conclude that all the forms $V_2f$, where $f$ is a weight $1$ form of level $N$, are classical forms, and so the dimension of the space $S_2(\Gamma_0(N), \F_2)_{\mathfrak{m}}\cap\Ker A\theta$ is exactly the dimension $S_1(\Gamma_0(N), \F_2)^{\Katz}_{\mathfrak{m}}$.  And so from Lemma \ref{prelimind}, taking a direct sum over all $\mathfrak{m}$, we obtain Theorem \ref{mainthm}.
\section{Examples}\label{examples}
In this section we use Theorem \ref{mainthm} to make nontrivial observations about the index of $\T^{\an}$ inside $\T$.
\subsection{$N\equiv3\bmod4$}
\begin{lem}\label{example3mod4}If $N\equiv3\bmod4$ is prime, the anemic Hecke algebra $\T^{\an}$ is equal to the full algebra $\T$ if and only if the class group $\textup{Cl}(\Q(\sqrt{-N}))$ is trivial.\end{lem}
\begin{proof}If $K=\Q(\sqrt{-N})$ has class number greater than $1$, by genus theory, since the discriminant of $K$ is $-N$ which is divisible by only a single prime, the $2$-part of the class group of $K$ is trivial, so $\text{Cl}(K)$ has a nontrivial mod $2$ multiplicative character which translates to an unramified mod 2 character $\chi$ of $\Gal(\overline{\Q}/K)$.  Inducing this to $\Gal(\overline{\Q}/\Q)$, we get a dihedral representation with Artin conductor equal to $N$.  Wiese proves in \cite{MR2054983} that all dihedral representations give rise to Katz modular forms, and so the space $S_1(\Gamma_0(N), \F_2)^{\Katz}$ is nontrivial, and hence $\T^{\an}\subsetneq\T$.

This shows that if $N$ is not $3, 7, 11, 19, 43, 67$ or $163$ (and is still a $3\bmod 4$ prime), $\T^{\an}(N)\subsetneq\T(N)$.  On the other hand, for $N=3$ and $N=7$ there are no modular forms of weight $2$, and for the other $N$, computer verification using the techniques of modular symbols, such as described in \cite{MR2289048}, provides the following table:\newline$\begin{array}{c|c}N& T_2\\ \hline11&-2T_1 \\ \hline19&0 \\ \hline43&-2T_1-2T_3+T_5 \\ \hline67&T_3-T_{11} \\ \hline163&\begin{array}{c}30T_1-16T_3-23T_5-9T_7+18T_9+3T_{11}-24T_{13}\\ +12T_{15}+40T_{17}-16T_{19}-14T_{21}-9T_{23}+2T_{25}+32T_{27} \end{array}\\ \hline\end{array}$

These each prove that there are no Katz eigenforms of weight $1$ and level $N$ for any of these $N$, and in turn that there are no Galois representations that could provide such forms.  Of course, we knew \textit{a priori} there were no dihedral representations, as they would need to arise from the class group, but we now know that there are no larger-image representations.\end{proof}
\subsection{$N\equiv1\bmod4$}
\begin{ques}\label{lenstra}Is it true that for a positive proportion of prime $N\equiv1\bmod 4$, the anemic Hecke algebra $\T^{\an}$ is not equal to the full algebra $\T$, and for a positive proportion of $N$, $\T^{\an}$ is equal to $\T$?\end{ques}
We cannot immediately claim anything about the class group, because the Cohen-Lenstra heuristics \cite[C11]{MR756082} claim that approximately $75.446\%$ of positive prime-discriminant quadratic extensions have trivial class group, so that there can be no dihedral modular forms.

The strong form of Serre's conjecture due to Edixhoven \cite[Conjecture 1.8]{MR1638480} is not known, where the strong form differs from the form proven by Khare and Wintenberger in \cite{MR2551763} in this weight $1$ case.  A result of Wiese for dihedral representations \cite{MR2054983} is known, and a converse (that the corresponding representation $\overline{\rho}$ is unramified at $2$) has been proven \cite[Corollary 1.3]{MR3247800}.  We may also use Theorem \ref{mainthm} to construct weight $1$ forms in the case that the eigenvalues of $\Frob_2$ in the characteristic $2$ representation are distinct, because there are two possible values for $a_2$, implying that $\T_{\mathfrak{m}}\neq\T^{\an}_{\mathfrak{m}}$.

We also know the subgroups of $\SL_2(\overline{\F}_2)$, by Dickson, of four types: cyclic, upper-triangular, dihedral, and full-image (see \cite[Chapter 3, Theorem 6.17]{MR514842}).  We know a modular representation must be absolutely irreducible: if not, say $f$ is a weight $1$ form for which $\overline{\rho}_f$ is reducible.  Then $Af$ is a weight $2$ form with the same representation, along with $Vf$ in the same generalized eigenspace.  But in Section \ref{reducible} we proved that $T_2$ is already contained in the Hecke algebra corresponding to any eigenform with reducible representation, meaning that the dimension of $S_2(\Gamma_0(N),\F_2)_{\mathfrak{m}}$ is dimension $1$, not $2$.  Therefore only absolutely irreducible representations can be modular, so only dihedral and full-image representations can exist.  So assuming the strong version of Serre's conjecture, we know that for any weight $1$ forms to exist at level $N$, we need either a dihedral extension of $\Q$, which must arise from inducing from the class group of $\Q(\sqrt{N})$, or we need an extension of $\Q$ unramified outside $N$ with Galois group isomorphic to $\SL_2(\F_{2^k})$ for some $k$.

Work has been done by Lipnowski \cite{lipnowski2016bhargava} to interpret Bhargava's heuristics for the Galois group $\GL_2(\F_p)$ for $p$ a prime, in order to count elliptic curves by their conductors through their $p$-adic representations.  Although not done in this current note, it appears tractable to similarly analyze the groups $\SL_2(\F_{2^k})$ and obtain a heuristic, explicit or not, on how many primes $p$ have an elsewhere-unramified extension with each of these as their Galois groups.  Because of the Cohen-Lenstra heuristics, it appears likely that infinitely many, even a positive proportion, of primes $1\bmod4$ have no weight $1$ forms, so $\T=\T^{\an}$, and a positive proportion of primes have some weight $1$ form so $\T^{\an}\subsetneq\T$.

\subsubsection{Explicit example: $N=653$}
An instructive example is that of $N=653$.  Of course this is $1\bmod 4$, and so any dihedral representation that would give a weight $1$ form would have to come from an induction of the class group of $\Q(\sqrt{653})$, but the Minkowski bound is $\frac{1}{2}\sqrt{653}\approx12.77$, and $2,3,5$ are inert and $7=230^2-653\cdot9^2$ and $-11=51^2-653\cdot2^2$ are norms of principal ideals.  So $\Q(\sqrt{653})$ has class number $1$.  But the Galois closure $L$ of the field $\Q[x]/(x^5+3x^3-6x^2+2x-1)$ has Galois group $A_5=\SL_2(\F_4)$, and is ramified only at $653$ with ramification degree $2$ and inertial degree $2$.  Therefore, Edixhoven predicts that the tautological Galois representation gives rise to a weight $1$ level $\Gamma_0(653)$ modular form.  This is not a classical form, as $\SL_2(\F_4)$ does not embed into $\GL_2(\C)$, where all weight $1$ characteristic $0$ eigenforms must arise from.

On the other hand, $\SL_2(\F_4)$ does embed into $\PGL_2(\C)$, and by a theorem of Tate, all projective Galois representations lift.  We can follow the proof given by Serre in \cite{MR0450201} to obtain a lift, unramified away from $653$, and with Artin conductor $653^2$.  The fixed field of the kernel of this representation is a quadratic extension of $L[x]/(x^4-x^3+82x^2-1102x+13537)$, which is itself the compositum of $L$ and the quartic subfield of the $653$rd roots of unity.  Locally at $653$ it is a faithful representation of $\Gal(\Q_{653}(\sqrt[8]{653},\sqrt{2})/\Q_{653})$, a Galois group isomorphic to $\langle x,y|x^8=y^2=e, yx=x^5y\rangle$.

We therefore find that, as the Artin conjecture for odd representations has been proven in \cite{MR2551763}, an eigenform of weight $1$ and level $653^2$ that reduces to the characteristic $2$ form of level $653$ we found above.  We can additionally twist by the nontrivial character of $\Q(\sqrt{653})/\Q$, not changing the determinant or level, to get a second Artin representation, and hence a second modular form of the same weight and nebentypus.  These two eigenforms are congruent mod $2$, so their average is also an integral form, and there is therefore a nilpotent element of the weight $1$ mod $2$ Hecke algebra, in a similar sense to \cite[Lemma 3.8]{MR2460912}.  And conjugating the $\F_4$-forms, we obtain $2$ more weight $1$ forms of level $653$.  So the index of $\T^{\an}$ in $\T$ must be at least $16$.

Indeed, we can find the following four (non-eigen)forms of weight $2$ and level $653$:\begin{align*}f_1=0 &q^{1} &+ 1 &q^{2} &+ 2 &q^{3} &- 4 &q^{4} &+ 0 &q^{5} &+ 2 &q^{6} &+ 0 &q^{7} &+ 4 &q^{8} &+ 0 &q^{9} &+ 4 &q^{10} &+ 0 &q^{11} &+ 1 &q^{12} &- 6 &q^{13} &+\ldots\\
f_2=0 &q^{1} &+ 0 &q^{2} &+ 2 &q^{3} &- 3 &q^{4} &+ 0 &q^{5} &+ 2 &q^{6} &+ 2 &q^{7} &+ 2 &q^{8} &+ 4 &q^{9} &- 3 &q^{10} &+ 4 &q^{11} &- 6 &q^{12} &+ 0 &q^{13} &+\ldots\\
f_3=0 &q^{1} &+ 0 &q^{2} &+ 0 &q^{3} &+ 4 &q^{4} &+ 0 &q^{5} &+ 1 &q^{6} &+ 2 &q^{7} &+ 2 &q^{8} &+ 4 &q^{9} &+ 5 &q^{10} &+ 2 &q^{11} &+ 0 &q^{12} &+ 4 &q^{13} &+\ldots\\
f_4=0 &q^{1} &- 2 &q^{2} &- 6 &q^{3} &+ 2 &q^{4} &+ 0 &q^{5} &+ 2 &q^{6} &+ 2 &q^{7} &- 5 &q^{8} &+ 0 &q^{9} &+ 0 &q^{10} &- 2 &q^{11} &- 6 &q^{12} &- 2 &q^{13} &+\ldots\end{align*}each of whose odd-power coefficients are all even, proving that none of $T_2, T_4, T_6$ or $T_8$ are in $\T^{\an}$ plus the other $3$.  But a calculation up to the Sturm bound of $109$ proves that there are no other modular forms with all odd-power coefficients and coefficients of $q^2, q^4, q^6, q^8$ all even but some other coefficient is odd.  Therefore $\T=2\T+\langle T_2, T_4, T_6, T_8\rangle$, so $\T/2\T$ is generated as an $\F_2$-vector space by $T_2, T_4, T_6, T_8$.  By Lemma \ref{doubleT}, $\T/\T^{\an}$ is a quotient of $\T/2\T$, but from the above forms $T_2, T_4, T_6, T_8$ are independent in $\T/\T^{\an}$ so the index of $\T^{\an}$ in $\T$ must be exactly $2^4=16$.
\bibliographystyle{alpha}
\bibliography{indexofTbib}
\end{document}